\documentclass[12pt,notitlepage,a4paper,reqno,ps]{amsart}
\usepackage{eurosym}
\usepackage{amssymb}
\usepackage{amstext}
\usepackage{color}
\usepackage{xcolor}
\usepackage{amsmath,amsthm,amssymb,amscd,epsfig,esint, mathrsfs,graphics,color}
 \usepackage{wrapfig}
\usepackage{verbatim}
\usepackage[latin1]{inputenc}
\xdefinecolor{azzurro}{rgb}{0.2,0.9,0.9}
\xdefinecolor{violetto}{rgb}{0.5,0.5,0.9}
\xdefinecolor{marrone}{rgb}{0.8,0.4,0.2}

\newcommand{\loc}{\textnormal{loc}}
\newcommand{\medint}{-\kern  -,375cm\int}

\definecolor{ora}{rgb}{0.8,0.2,0.1}
\definecolor{vio}{rgb}{0.5,0,0.5}
\definecolor{gre}{rgb}{0.1,0.6,0}
\definecolor{verde}{rgb}{0,0.7,0.4}
\newenvironment{michelarev}{\color{azzurro}}{\color{black}}
\newcommand{\bmicr}{\begin{michelarev}}
\newcommand{\emicr}{\end{michelarev}}
\theoremstyle{plain}
\newtheorem{theorem}{Theorem}[section]

\newtheorem{lemma}[theorem]{Lemma}
\newtheorem{proposition}[theorem]{Proposition}
\theoremstyle{definition}

\theoremstyle{remark}

\theoremstyle{plain}

\def\R{\mathbb{R}}
\numberwithin{equation}{section} \makeatletter

\renewcommand{\p@enumi}{\thesection.}
\makeatother  \allowdisplaybreaks
\email{mascolo@unifi.it}
  \email{antonia.passarellidinapoli@unina.it}
\keywords{Higher differentiability, subquadratic problems, non standard growth.}
\subjclass[2000]{35J87, 49J40; 47J20}

\begin{document}
\title[Higher differentiability for   problems under $p,q$ subquadratic  growth ]{Higher differentiability for a class of  problems\\ under $p,q$ subquadratic  growth}
\author[E. Mascolo -- A. Passarelli di Napoli]{E. Mascolo -- Antonia Passarelli di Napoli}

\address{Dipartimento di Matematica "U. Dini", Universit\`a degli Studi di Firenze,
viale  Morgagni , 41125 Firenze, Italy}
\address{Dipartimento di Matematica e Applicazioni ``R. Caccioppoli''
\\
Universit\`a degli Studi di Napoli
\\
``Federico II''
Via Cintia, 80126, Napoli, Italy}
\thanks{\textit{Acknowledgements.}
The work of the authors is supported by GNAMPA (Gruppo Nazionale per l'Analisi Matematica, la Probabilit\`a e le loro Applicazioni) of INdAM (Istituto Nazionale di Alta Matematica). A. Passarelli di Napoli has been supported by Università di Napoli $\lq\lq$Federico II" through the project
FRA--000022-ALTRI-CDA-752021-FRA-PASSARELLI.} 

\begin{abstract}
We study the higher differentiability for  nonlinear elliptic equation in divergence form $\mathcal{A}(x,Du)=b(x)$.
The result covers the cases in which $\mathcal{A}(x, \xi)$ satisfies $p,q$ growth, with $1<p<2$ in $\xi$ and a Sobolev dependence of  with respect to $x$. By means of an  a-priori estimate we ensure the $W^{2,p}_{\mathrm{loc}}(\Omega)$-property for the solution of the boundary value problem.

\end{abstract}

\maketitle

\begin{center}
\fbox{\today}
\end{center}

\section{Introduction}

The paper deals with the following boundary value problem for elliptic equations
\begin{equation}\label{prob}
	\begin{cases}
		\mathrm{div}\,(\mathcal{A}(x,Du))=b(x)\qquad\qquad \text{in}\,\,\Omega\cr\cr
		u=u_0\qquad\qquad\qquad\qquad\qquad \text{on}\,\,\,\partial\Omega
	\end{cases}
\end{equation} 
where $\Omega\subset\mathbb{R}^n$ is a bounded open set and the operator $\mathcal{A}:\Omega\times\mathbb{R}^n\to \mathbb{R}$ is a Carath\'edory map satisfying the following set of assumptions for a couple of exponents $p,q$ such that $1<p<q$ and $1<p\le 2$
\begin{equation}\label{(A2)}
	\langle \mathcal{A}(x,\xi)-\mathcal{A}(x,\eta),\xi-\eta\rangle\ge \nu|\xi-\eta|^2(\mu^2+|\xi|^2+|\eta|^2)^{\frac{p-2}{2}}
\end{equation}
\begin{equation}\label{(A1)}
	| \mathcal{A}(x,\xi)-\mathcal{A}(x,\eta)|\le L|\xi-\eta|\left[(1+|\xi|^2+|\eta|^2)^{\frac{q-2}{2}}+{(\mu^2+|\xi|^2+|\eta|^2)^{\frac{p-2}{2}}}\right]
\end{equation}
\begin{equation}\label{(A3)}
	| \mathcal{A}(x,\xi)-\mathcal{A}(y,\xi)|\le |x-y|(k(x)+k(y))(1+|\xi|^2)^{\frac{q-1}{2}}
\end{equation}
where $L,\nu$ are fixed constants, $\mu\in [0,1]$ is a parameter and $k\in L^r_{\mathrm{loc}}(\Omega)$ with $r>n$ is a non negative function.
\\
Let us observe that if we assume that
$$\mathcal{A}(x,0)=0,$$
assumption \eqref{(A1)} implies
\begin{equation}\label{(A1')}
	| \mathcal{A}(x,\xi)|\le L(1+|\xi|^2)^{\frac{q-1}{2}}
\end{equation}
Then, setting $f(x)=\mathcal {A}(x,0) \in L^{\infty}$, we have
\begin{equation*}\label{(A1'')}
	| \mathcal{A}(x,\xi)|\le L(1+|\xi|^2)^{\frac{q-1}{2}}+f(x)
\end{equation*}

To simplify the presentation, from now on, we shall assume that \eqref{(A1')} is in force.
\\
Assumption \eqref{(A3)}, by virtue of the characterization of the Sobolev function due to Hajlasz, \cite{H}, implies that the partial map $x\mapsto \mathcal{A}(x,\xi)$ belongs to the Sobolev space $W^{1,r}_{\mathrm{loc}}(\Omega)$ and that there exists $\tilde{k}\in L^{r}_{\mathrm{loc}}(\Omega)$ such that
\begin{equation}\label{assob}
	|D_x\mathcal{A}(x,\xi)|\le \tilde k(x)(1+|\xi|^2)^{\frac{q-1}{2}}.
\end{equation}

We prove the following  existence and higher differentiability result.
\begin{theorem}\label{main1}
 Assume that $\mathcal{A}(x,\xi)$ satisfy \eqref{(A2)}---\eqref{(A1')} with $1<p<q$ and {$1<p\le 2$}
 such that
\begin{equation}\label{gapa}
	\frac{q}{p}<1+\min\left\{\frac{1}{n}-\frac{1}{r},\, \frac{2(p-1)}{p(n-2)}\right\} \qquad \text{if}\,\, n>2
\end{equation}
or if  $n=2$
\begin{equation}\label{gapa2}
	\frac{q}{p}<1+\frac{1}{2}-\frac{1}{r} \qquad\qquad\qquad 
\end{equation}
If $$u_0\in W^{1,\frac{p(q-1)}{p-1}}(\Omega)\,\,\, \text{and}\,\,\, b\in L^{\frac{p}{p-1}}_{\mathrm{loc}}(\Omega),$$
then problem \eqref{prob} admits a solution $v\in \big(u_0+W^{1,p}_0(\Omega)\big) \cap W^{2,p}_{\mathrm{loc}}(\Omega)$. Moreover    setting $V_p(Du)=(\mu^2+|Du|^2)^{\frac{p-2}{4}}Du$, the following estimates
\begin{eqnarray}\label{dersec02}
	\int_{B_{\frac{R}{2}}}|D(V_p(D v))|^2\,dx\le c(1+||k||_{L^r(B_R)}+ ||b||_{L^{p'}(B_R)})^{\gamma}\left(1+\int_{B_R}  (1+|Du|)^{p}\,dx\right)^{\gamma}\end{eqnarray}
	
\begin{eqnarray}\label{dersec0}
\int_{B_{\frac{R}{2}}}|D^2 v|^p\,dx\le c(1+||k||_{L^r(B_R)}+ ||b||_{L^{p'}(B_R)})^{\gamma}\left(1+\int_{B_R}  (1+|Du|)^{p}\,dx\right)^{\gamma}\end{eqnarray}
hold for every ball $B_R\Subset\Omega$, for an exponent $\gamma=\gamma(p,q,n)$ and  $c=c(\nu,L,n,p,q,R)$.
\end{theorem}


The novelty of Theorem \ref{main1} is in two directions 
\vskip.2cm
{First, we consider operators with  $p,q$ growth, which means that  the  ellipticity of the leading part of $\xi\mapsto \mathcal{A}(x,\xi)$ has $p$-growth (see \eqref{(A2)}) while it satisfies a   $q$- bound (see \eqref{(A1)}), with $1<p<q$ and $1<p<2$.
Roughly speaking, we are dealing with the subquadratic non standard growth case which means that the problem under consideration can be singular and degenerate.}
\\
{In the last few years, the study of the regularity properties of solution and minimizers of problems with non standard growth conditions has undergone remarkable  developments, motivated in part by the applications.
We would mention the first papers of Marcellini  \cite{mar89, mar91, mar93} and the  recent \cite{mar20-1, mar20-2, mar20-3}, the result on higher integrability and differentiability  
in \cite{elm1, elm2, lms2, Sbordone} and more recently \cite{ CKP, CKP1, cupleomas, cupgiagiopass, cupmarmaspass, EMM1, EMM2, EMM3, EMM4}. 
{For complete details and references on problems with non standard growth we refer to the recent surveys \cite{ mar20-2, MR}} 
\\
It is well known that a restriction between $p$ and $q$ is necessary by virtue of  the celebrated counterexample by Marcellini (see \cite{mar91})}.
\vskip.2cm
The second principal feature is the Sobolev dependence of $\mathcal{A}(x,\xi)$  with respect to $x$ (see \eqref{(A3)}).

Recently, there has been an increasing interest in the study of the regularity under this assumptions on the function that measures the oscillation of the operator $\mathcal{A}(x,\xi)$ with respect to the $x$-variable.
The degeneracy of $\mathcal{A}(x,\xi)$ has the same nature of 
$$\mathrm{div}\big(|Du|^{p-2}Du+a(x)|Du|^{q-2}Du\big)=0$$
related to the double phase functional
\begin{equation}\label{double}
	\mathcal{I}(u)=\int_\Omega |Du|^p+a(x)|Du|^q
\end{equation}
which  has been intensively studied starting from the papers  \cite{colmin, colmin2}. For references on the regularity properties of minimizers to \eqref{double}, see  also \cite{BM} and \cite{DM}. 

Actually, it is now  completely clear that  the weak differentiability of the partial map $\xi\to \mathcal{A}(x,\xi)$ leads to an higher differentiability of the gradient of the 
solutions (see for example \cite{cupgiagiopass, Gentile, Gentile2, giova1, giova2, passarelli1, passarelli2}. 

In \cite{EMM3,EMM4} for integral functionals under $p,q$,  growth with Sobolev coefficients in $L^r$, $r>n$,  it has been shown that the bound $\frac{q}{p}<1+\frac{1}{n}-\frac{1}{r}$ permit to obatin the local Lipschitz continuity of the minimizers.

 {Moreover, it is worth mentioning that the bound at \eqref{gapa} has been already used in \cite{DM} (see also \cite{BM}) for the study of problems with subquadratic non standard growth conditions.

In the standard growth, i.e. $1<p=q<2$ in assumption \eqref{(A1)}--\eqref{(A3)},  and Lipschitz continuous coefficients, i.e. $k\in L^\infty_{\mathrm{loc}}(\Omega)$ in assumption \eqref{(A3)}, the regularity is due to Tolksdorf (\cite{Tolksdorf}) (see also \cite[Chapter 8]{Giusti}), for the case of vector-valued minimizers  see \cite{af}. 

In case of $p,q$ growth with  $2\leq p<q$, in \cite{mar91, mar20-1, mar20-2} Marcellini established the $W^{2,2}_{\mathrm{loc}}$ regularity for local weak solution $u \in W^{1,q}(\Omega)$.

More recently, in the  $p$-growth case, $1<p<2$ and $k\in L^r_{\mathrm{loc}}(\Omega)$ with $r\ge n$ in  \eqref{(A3)}, an higher differentiability result 
for local minimizers of integral functionals has been established in \cite{Gentile, Gentile2}.

A first main step in the proof of our main result   is 
 an a-priori estimate for the $W^{2,p}$-norm of the weak local solution  $u$ of 
 of the equation 
$$\mathrm{div}\,(\mathcal{A}(x,Du))=b(x)$$
i.e. $ u \in W^{1,q}_{\mathrm{loc}}$    such that   
$$
\int \mathcal{A}(x,Du) D\varphi(x)\, dx= \int b(x) \varphi(x) \, dx\qquad \text{for all}\,\,\varphi \in C^{\infty}_0(\Omega).
$$
{We prove that if  $Du\in L^{\frac{2^*}{2}p}$} then it   locally belongs to $W^{2,p}$ and  the norm of its second derivatives can be estimated by the $L^p$ norm of its gradient.
\\
Actually, as far as we know, our a-priori estimate is a first  regularity result for weak local solutions of an equation under $p,q$-growth $1<p<q$, $1<p<2$, and a Sobolev assumption on the partial map $x\mapsto \mathcal{A}(x,\xi)$ (see \eqref{(A3)}). 

Then we give an existence and regularity result for  the Dirichlet boundary problem \eqref{prob}.

More precisely,  the solution to  problem \eqref{prob} which belongs to $W^{2,p}$ locally is constructed as a $W^{1,q}$-limit of a sequence of solution to regular equations with standard $q$-growth of the form
$$\mathcal{A}_\varepsilon(x,\xi)=\tilde{\mathcal{A}}_\varepsilon(x,\xi)+\varepsilon(1+|\xi|^2)^{\frac{q-2}{2}}\xi,$$

where $\tilde{\mathcal{A}}_\varepsilon(x,\xi)$ is the regularized of $\mathcal{A}(x,\xi)$  with respect to the variable $x$, i.e.
$$\tilde{\mathcal{A}}_\varepsilon(x,\xi)=\int_{B_1(0)}\phi(\omega)\mathcal{A}(x+\varepsilon\omega,\xi)\,d\omega,$$
with $\phi$ a smooth mollifier. 

Each problem  
has a unique smooth  solution $v_{\varepsilon} \in W^{1,q}_0+u_0$ in view of the available existence and regularity results.

Then, we apply to the sequence of approximating solutions the a-priori estimate which takes into account only the assumptions \eqref{(A1)}--\eqref{(A3)} and therefore is independent of $\varepsilon$. Therefore we obtain a uniform control  for the norm of $v_{\varepsilon}$ in  $W^{2,p}$. By passing to the limit we obtain a solution $v$ of \eqref{prob} with higher differentiability properties.
\\
Let us remark that in subquadratic $p,q$- growth case, in \cite{lms} a  higher integrability result has been established  when$$\frac{2n}{n+2}\le p\le q\le 2.$$
It is worth to observe that if $\frac{2n}{n+2}\le p$, the restriction on the gap \eqref{gapa} reduces to
$$\frac{q}{p}<1+\frac{1}{n}-\frac{1}{r}$$
which is the sharp one found in \cite{elm1} to establish the local Lipschitz continuity of minimizers of degenerate functionals with $(p,q)$ growth conditions and Sobolev coefficients. Indeed
$$p\ge \frac{2n}{n+2}\,\,\Longrightarrow\,\, \frac{p-1}{p}\ge \frac{n-2}{2n}\,\,\Longrightarrow\,\, \frac{2(p-1)}{p(n-2)}\ge \frac{1}{n}$$
and so 
$$\min\left\{\frac{1}{n}-\frac{1}{r},\, \frac{2(p-1)}{p(n-2)}\right\}=\frac{1}{n}-\frac{1}{r}.$$

%

\vskip 2cm
\section{Preliminary results}
\noindent
In this paper we shall denote by $C$ or $c$  a
general positive constant that may vary on different occasions, even within the
same line of estimates.
Relevant dependencies  will be suitably emphasized using
parentheses or subscripts.  In what follows, $B(x,r)=B_r(x)=\{y\in \R^n:\,\, |y-x|<r\}$ will denote the ball centered at $x$ of radius $r$.
We shall omit the dependence on the center and on the radius when no confusion arises.


The following lemma has important applications in the so called
hole-filling method. Its proof can be found for example in
\cite[Lemma 6.1]{Giusti} .
\medskip
\begin{lemma}\label{iter} Let $h:[r, R_{0}]\to \mathbb{R}$ be a nonnegative bounded function and $0<\vartheta<1$,
$A, B\ge 0$ and $\beta>0$. Assume that
$$
h(s)\leq \vartheta h(t)+\frac{A}{(t-s)^{\beta}}+B,
$$
for all $r\leq s<t\leq R_{0}$. Then
$$
h(r)\leq \frac{c A}{(R_{0}-r)^{\beta}}+cB ,
$$
where $c=c(\vartheta, \beta)>0$.
\end{lemma}
We  introduce the usual notation
\begin{eqnarray*}
	2^*=\begin{cases} \frac{2n}{n-2}\qquad\qquad\qquad\qquad\qquad\qquad \text{if}\,\, n>2\\
		\text{any finite exponent}\,\, t>2 \qquad\,\,\,\text{if}\,\, n=2
	\end{cases}
\end{eqnarray*}

\noindent For $p>1$ and $\mu\in [0,1]$, let us define
\begin{equation}\label{Vi}
	V_p(\xi)=(\mu^2+|\xi|^2)^{\frac{p-2}{4}}\xi, \qquad \xi\in \mathbb{R}^n
\end{equation}
We shall use the following estimates, whose proof can be found in \cite{af} (see also \cite[Step 2]{mar85}).
\begin{lemma}\label{Vi}
Let $1<p<2$. There exists a constant $c=c(n,p)>0$
such that, for any $\xi$, $\eta \in \R^n$, $\xi\not=\eta$, it holds
\begin{equation}\label{VIes}
c^{-1}\Bigl(\mu^2 +| \xi |^2+| \eta |^2 \Bigr)^{\frac{p-2}{2}}\leq
\frac{|V_{p}(\xi )-V_{p}(\eta )|^2}{|\xi -\eta |^2} \leq
c\Bigl( \mu^2 +|\xi |^2+|\eta |^2 \Bigr)^{\frac{p-2}{2}}	
\end{equation}
\end{lemma}

\begin{lemma}\label{rem}
Let $1<p<2$ and $u\in W^{1,1}_{\mathrm{loc}}(\Omega)$. Then 
	\begin{equation*}
		V_p(Du)\in L^{2^*}_{\loc}(\Omega)\quad\Longrightarrow\quad Du\in L^{\frac{2^*p}{2}}_{\loc}(\Omega).
	\end{equation*}
\end{lemma}	  
	 
\begin{proof}	Note that the thesis is obvious if $\mu=0$.
	In case $\mu>0$, we have
	\begin{eqnarray*}
		&&\int_{B_R}|Du|^{\frac{2^*p}{2}}\,dx=\int_{B_R}\left||Du|^{\frac{p}{2}-1}Du\right|^{2^*}\,dx\cr\cr
		&=&\int_{\{x\in B_R:\, |Du|\le \mu\}}\left||Du|^{\frac{p}{2}-1}Du\right|^{2^*}\,dx+\int_{\{x\in B_R:\, |Du|> \mu\}}\left||Du|^{\frac{p}{2}-1}Du\right|^{2^*}\,dx\cr\cr
		&=&\int_{\{x\in B_R:\, |Du|\le \mu\}}\left||Du|^{\frac{p}{2}-1}Du\right|^{2^*}\,dx+\int_{\{x\in B_R:\, |Du|> \mu\}}\left|\left(\frac{|Du|^2}{2}+\frac{|Du|^2}{2}\right)^{\frac{p-2}{4}}|Du|\right|^{2^*}\,dx\cr\cr
		&\le&\mu^{\frac{2^*p}{2}}|B_R|+\int_{\{x\in B_R:\, |Du|> \mu\}}\left|\left(\frac{\mu^2}{2}+\frac{|Du|^2}{2}\right)^{\frac{p-2}{4}}|Du|\right|^{2^*}\,dx,
	\end{eqnarray*}
	since $p-2<0$. Therefore, we have
	\begin{eqnarray}\label{stinorma}
		\int_{B_R}|Du|^{\frac{2^*p}{2}}\,dx&=&\le \mu^{\frac{2^*p}{2}}|B_R|+\int_{\{x\in B_R:\, |Du|> \mu\}}\left|V_p(Du)\right|^{2^*}\,dx\cr\cr
		&\le & \int_{B_R}\left(1+\left|V_p(Du)\right|^{2^*}\right)\,dx<+\infty,
	\end{eqnarray}
	where we also used that $\mu\le 1$.
\end{proof}

\begin{lemma}\label{VD}
Let $1<p<2$ and  $u\in W^{1,1}_{\mathrm{loc}}(\Omega)$. 	Then
$$V_p(Du)\in W^{1,2}_{\mathrm{loc}}(\Omega)\quad\Longrightarrow\quad u\in W^{2,p}_{\mathrm{loc}}(\Omega)$$
\end{lemma}
\begin{proof}
	By the use of H\"older's inequality we get
	\begin{eqnarray*}
	\int_{B_R}|D^2u|^p&=&\int_{B_R}|D^2u|^p(\mu^2+|Du|^2)^{\frac{p(p-2)}{4}}(\mu^2+|Du|^2)^{\frac{p(2-p)}{4}}\cr\cr
	&\le&\left(	\int_{B_R}|D^2u|^2(\mu^2+|Du|^2)^{\frac{p-2}{2}}\right)^{\frac{p}{2}}\left(	\int_{B_R}(\mu^2+|Du|^2)^{\frac{p}{2}}\right)^{\frac{2-p}{2}}\cr\cr
	&\le &c\left(	\int_{B_R}|D(V_p(Du))|^2\right)^{\frac{p}{2}}\left(	\int_{B_R}(\mu^2+|Du|^2)^{\frac{p}{2}}\right)^{\frac{2-p}{2}}
	\end{eqnarray*}

\end{proof}

\subsection{Difference quotient}
\medskip
\noindent
We recall some properties of the finite difference
operator  needed in the sequel.
 We start with the description of some elementary properties that can be found, for example, in \cite{Giusti}.

\begin{proposition}\label{findiffpr}

Let $f$ and $g$ be two functions such that
$f, g\in W^{1,p}(\Omega)$, with $p\geq 1$, and let us consider the set
$$
\Omega_{|h|}:=\left\{x\in \Omega : dist(x, \partial\Omega)>|h|\right\}.
$$
Then
\begin{itemize}
\item[$(d1)$] $\tau_{h}f\in W^{1,p}(\Omega)$ and
$$
D_{i} (\tau_{h}f)=\tau_{h}(D_{i}f).
$$
\item[$(d2)$] If at least one of the functions $F$ or $G$ has support contained
in $\Omega_{|h|},$ then
$$
\int_{\Omega} f\, \tau_{h} g\, dx =-\int_{\Omega} g\, \tau_{-h}f\, dx.
$$
\item[$(d3)$] We have
$$
\tau_{h}(fg)(x)=f(x+h)\tau_{h}g(x)+g(x)\tau_{h}f(x).
$$
\end{itemize}
\end{proposition}
\noindent The next result about finite difference operator is a kind of integral version
of Lagrange Theorem.
\begin{lemma}\label{le1} If $0<\rho<R$, $|h|<\frac{R-\rho}{2}$, $1 < p <+\infty$,
 and $f, Df\in L^{p}(B_{R})$ then
$$
\int_{B_{\rho}} |\tau_{h} f(x)|^{p}\ dx\leq c(n,p)|h|^{p} \int_{B_{R}}
|D f(x)|^{p}\ dx .
$$
Moreover
$$
\int_{B_{\rho}} |f(x+h )|^{p}\ dx\leq  \int_{B_{R}}
|f(x)|^{p}\ dx .
$$
\end{lemma}
\noindent Now, we recall the fundamental Sobolev embedding property.
\begin{lemma}\label{lep} Let $f:\mathbb{R}^{n}\to\mathbb{R}^{N}$, $f\in
L^{p}(B_{R})$ with $1<p<n$. Suppose that there exist $\rho\in(0,R)$ and  $M>0$ such that
$$
\sum_{s=1}^{n}\int_{B_{\rho}}|\tau_{s,h}f(x)|^{p}\,dx\leq
M^{p} |h|^{p},
$$
for every $h$ with $|h|<\frac{R-\rho}{2}$. Then $$f\in
W^{1,p}(B_{\rho})\cap L^{\frac{np}{n-p}}(B_{\rho}).$$ Moreover
\begin{equation}\label{frac0}
	||Df||_{L^{p}(B_{\rho})}\leq
M
\end{equation}
and
\begin{equation}\label{frac}
	||f||_{L^{\frac{np}{n-p}}(B_{\rho})}\leq
c\left(M+||f||_{L^{p}(B_{R})}\right)
\end{equation}
with $c= c(n, N,p)$.
\end{lemma}
For the proof see, for example, \cite[Lemma 8.2]{Giusti}.
\\
 
\medskip
\vskip2cm
\section{Proof of the a-priori estimate}
This section is devoted to the proof of an a-priori estimate, which is at the same time a first regularity result,  for local solution to the following equation
\begin{equation}\label{equa}
	\mathrm{div}(\mathcal{A}(x,Du))=b(x) \qquad\qquad \text{in}\,\,\,\Omega
\end{equation}
where $\Omega\subset \mathbb{R}^n$ is a bounded open set and the operator $\mathcal{A}: \Omega\times \mathbb{R}^{N\times n}\to \mathbb{R}$ is a Carath\'eodory map satisfying the following set of assumptions for a couple of exponents $1<p<q$  with $1<p\le 2$

\begin{equation}\label{(A2b)}
	\langle \mathcal{A}(x,\xi)-\mathcal{A}(x,\eta),\xi-\eta\rangle\ge \nu|\xi-\eta|^2(\mu^2+|\xi|^2+|\eta|^2)^{\frac{p-2}{2}}
\end{equation}

\begin{equation}\label{(A1b)}
	| \mathcal{A}(x,\xi)-\mathcal{A}(x,\eta)|\le L|\xi-\eta|\left[(1+|\xi|^2+|\eta|^2)^{\frac{p+q-4}{4}}+(\mu^2+|\xi|^2+|\eta|^2)^{\frac{p-2}{2}}
\right]
\end{equation}

\begin{equation}\label{(A3b)}
	| \mathcal{A}(x,\xi)-\mathcal{A}(y,\xi)|\le |x-y|(k(x)+k(y))(1+|\xi|^2)^{\frac{p+q-2}{4}}
\end{equation}
where $c_1,c_2,L,\nu>0$ are fixed constants, $\mu\in [0,1]$ is a parameter, $k\in L^r$ with $r>n$ is a non negative function.
\\
We define  a local solution to \eqref{equa} as a function $u\in W^{1,q}_{\mathrm{loc}}(\Omega)$  that satisfies the following integral identity
\begin{eqnarray}\label{locale}
\int_{\Omega}\langle \mathcal{A}(x,Du), D\varphi \rangle\,dx=\int_{\Omega} b(x)\varphi(x)\,dx\qquad \text{for all}\,\,\varphi\in C^{\infty}_0(\Omega).	
\end{eqnarray}
Our next aim is to prove the following

\begin{theorem}\label{apriori}
Let $u\in W^{1,q}_{\mathrm{loc}}(\Omega)$ be a local solution to \eqref{equa} under the assumptions \eqref{(A2b)}--\eqref{(A3b)} such that 
\begin{equation}\label{ipoapriori}
Du\in L^{\frac{2^*p}{2}}_{\mathrm{loc}}(\Omega).	
\end{equation} 
Assume moreover that
\begin{equation}\label{gap}
	\frac{q}{p}<1+\min\left\{2\left(\frac{1}{n}-\frac{1}{r}\right),\, \frac{4(p-1)}{p(n-2)}\right\}\qquad \text{if}\,\, n>2
\end{equation}
or
\begin{equation}\label{gapn=2}
	\frac{q}{p}<1+2\left(\frac{1}{n}-\frac{1}{r}\right),\qquad\qquad\qquad\qquad\qquad \text{if}\,\, n=2.
\end{equation}
Then  $$u\in W^{2,p}_{\mathrm{loc}}(\Omega),\qquad V_p(Du)\in W^{1,2}_{\mathrm{loc}}(\Omega)$$ and the following estimates 
\begin{eqnarray}\label{dersec2}
	\int_{B_{\frac{R}{2}}}|D(V_p(D u))|^2\,dx\le  c(1+||k||_{L^r(B_R)}+ ||b||_{L^{p'}(B_R)})^{\gamma}\left(1+\int_{B_R}  (1+|Du|)^{p}\,dx\right)^{\gamma},	
\end{eqnarray}
and
\begin{eqnarray}\label{dersec1}
\int_{B_{\frac{R}{2}}}|D^2 u|^p\,dx\le c(1+||k||_{L^r(B_R)}+ ||b||_{L^{p'}(B_R)})^{\gamma}\left(1+\int_{B_R}  (1+|Du|)^{p}\,dx\right)^{\gamma}	
\end{eqnarray}
hold for every ball $B_R\Subset\Omega$, for and exponent $\gamma=\gamma(p,q,n,r)$,  and  $c=c(\nu,L,n,p,q,r,R)$.
\end{theorem}
\begin{proof}
	Let $B_R\Subset\Omega$ be a ball, fix radii $0<\frac{R}{2}<s<t<R$ and a cut off function $\eta\in C^\infty_0(B_t)$ such that $\eta\equiv 1$ on $B_s$ and $|D\eta|\le \frac{C}{t-s}$.
	We  choose $\varphi=\tau_{-h}(\eta^2\tau_h u)\in W^{1,q}_{\mathrm{loc}}(\Omega)$ as  test function in 
	\eqref{locale}
	to deduce that
	$$\int_\Omega \langle \mathcal{A}(x,Du),\tau_{-h}(D(\eta^2\tau_hu))\rangle\,dx= \int_\Omega  b\tau_{-h}(\eta^2\tau_h u)\,dx$$
	i.e.
	$$\int_\Omega \langle \mathcal{A}(x,Du),\tau_{-h}(\eta^2\tau_hDu)+2\tau_{-h}(\eta D\eta \tau_h u)\rangle\,dx= \int_\Omega  b\tau_{-h}(\eta^2\tau_h u)\,dx.$$
	From previous equality, we get
	\begin{eqnarray*}
	\int_\Omega \langle \mathcal{A}(x,Du),\tau_{-h}(\eta^2\tau_hDu)\rangle\,dx&=&-2\int_\Omega \langle \mathcal{A}(x,Du),\tau_{-h}(\eta D\eta \tau_h u)\rangle\,dx\cr\cr
	&&\qquad+ \int_\Omega  b\tau_{-h}(\eta^2\tau_h u)\,dx.
	\end{eqnarray*}
	which, by virtue of (d2) in Proposition \ref{findiffpr}, implies
	\begin{eqnarray*}
	\int_\Omega \eta^2 \langle\tau_h( \mathcal{A}(x,Du)),\tau_hDu\rangle\,dx&=&-2\int_\Omega \langle \mathcal{A}(x,Du),\tau_{-h}(\eta D\eta \tau_h u)\rangle\,dx\cr\cr
	&&\qquad+ \int_\Omega  b\tau_{-h}(\eta^2\tau_h u)\,dx.
	\end{eqnarray*}
	Exploiting the definition of $\tau_h( \mathcal{A}(x,Du))$ in   previous equality, we have
	\begin{eqnarray*}
	&&\int_\Omega \eta^2 \langle  \mathcal{A}(x+h,Du(x+h))- \mathcal{A}(x+h,Du(x)),\tau_hDu\rangle\,dx\cr\cr
	&=&-\int_\Omega \eta^2 \langle  \mathcal{A}(x+h,Du(x))- \mathcal{A}(x,Du(x)),\tau_hDu\rangle\,dx\cr\cr	
	&&\qquad-2\int_\Omega \langle \mathcal{A}(x,Du),\tau_{-h}(\eta D\eta \tau_h u)\rangle\,dx\cr\cr
	&&\qquad+ \int_\Omega  b\tau_{-h}(\eta^2\tau_h u)\,dx.
	\end{eqnarray*}
Using the ellipticity assumption \eqref{(A2b)} in the left hand side and  assumptions \eqref{(A1b)} and \eqref{(A3b)}  in the right hand side of previous estimate, we obtain
\begin{eqnarray*}
	&&\nu\int_\Omega \eta^2(\mu^2+|Du(x)|^2+ |Du(x+h)|^2)^{\frac{p-2}{2}}|\tau_hDu|^2\,dx\cr\cr
	&\le& \int_\Omega\eta^2 | \mathcal{A}(x+h,Du(x))- \mathcal{A}(x,Du(x))|| \tau_h Du|\,dx\cr\cr
	&&+\quad 2\int_\Omega \eta | \mathcal{A}(x,Du(x))||\tau_{-h}(\eta D\eta \tau_h u)|\,dx\cr\cr
	&&\qquad+ \int_\Omega  |b||\tau_{-h}(\eta^2\tau_h u)|\,dx\cr\cr
	&\le& |h| \int_\Omega \eta^2 (k(x+h)+ k(x))(1+|Du(x)|^2)^{\frac{p+q-2}{4}}|\tau_h Du|\,dx\cr\cr
	&&+ c\int_\Omega (1+|Du(x)|)^{\frac{p+q-2}{2}}|\tau_{-h}(\eta D\eta \tau_h u)|\,dx\cr\cr
	&&\qquad+ \int_\Omega  |b||\tau_{-h}(\eta^2\tau_h u)|\,dx
	\end{eqnarray*}
	By the use of Young's inequality in the first integral of  the right hand side we get
	\begin{eqnarray*}
	&&\nu\int_\Omega \eta^2(\mu^2+|Du(x)|^2+ |Du(x+h)|^2)^{\frac{p-2}{2}}|\tau_hDu|^2\,dx\cr\cr
	&\le& \frac{\nu}{2}\int_\Omega \eta^2(\mu^2+|Du(x)|^2+ |Du(x+h)|^2)^{\frac{p-2}{2}}|\tau_hDu|^2\,dx \cr\cr
	&&\quad +c(\nu)|h|^2 \int_\Omega \eta^2 (k(x+h)+ k(x))^2(1+|Du(x)|^2+ |Du(x+h)|^2)^{\frac{q}{2}}\,dx\cr\cr
	&&+ c\int_\Omega (1+|Du(x)|)^{\frac{p+q-2}{2}}|\tau_{-h}(\eta D\eta \tau_h u)|\,dx\cr\cr
	&&\qquad+ \int_\Omega  |b||\tau_{-h}(\eta^2\tau_h u)|\,dx
	\end{eqnarray*}
	Reabsorbing the first integral in the right hand side,
	we get
	\begin{eqnarray}\label{stima}
	&&\int_\Omega \eta^2(\mu^2+|Du(x)|^2+ |Du(x+h)|^2)^{\frac{p-2}{2}}|\tau_hDu|^2\,dx\cr\cr
	&\le& c|h|^2 \int_\Omega \eta^2 (k(x+h)+ k(x))^2(1+|Du(x)|+ |Du(x+h)|)^{q}\,dx\cr\cr
	&&+ c\int_\Omega (1+|Du(x)|)^{\frac{p+q-2}{2}}|\tau_{-h}(\eta D\eta \tau_h u)|\,dx\cr\cr
	&&\qquad+ \int_\Omega  |b||\tau_{-h}(\eta^2\tau_h u)|\,dx
	=: I_1+I_2+I_3
	\end{eqnarray}
	The assumption $k\in L^r$ with $r>n$ and H\"older's inequality yields
	\begin{eqnarray}\label{I1}
	I_1 &\le&	c|h|^2\left(\int_{B_R}  k(x)^r\,dx\right)^{\frac{2}{r}}\left(\int_{B_t} \eta^2 (1+|Du(x)|)^{\frac{rq}{r-2}}\,dx\right)^{\frac{r-2}{r}}\cr\cr
	&=:& c|h|^2||k||_{L^r(B_R)}^2\left(\int_{B_t} \eta^2 (1+|Du(x)|)^{\frac{rq}{r-2}}\,dx\right)^{\frac{r-2}{r}}.
	\end{eqnarray}
	The a-priori assumption $Du\in L^{\frac{np}{n-2}}_{\mathrm{loc}}(\Omega)$ and  again H\"older's inequality imply
	\begin{eqnarray*}
	I_2+I_3
		&\le& c\left(\int_{B_t} (1+|Du(x)|)^{\frac{p}{p-1}\frac{p+q-2}{2}}\,dx\right)^{\frac{p-1}{p}}\left(\int_{B_t}|\tau_{-h}(\eta\nabla\eta \tau_h u)|^p\,dx\right)^{\frac{1}{p}}\cr\cr
	&&\qquad+ c\left(\int_{B_t} |b(x)|^{\frac{p}{p-1}}\,dx\right)^{\frac{p-1}{p}}\left(\int_{B_t}|\tau_{-h}(\eta^2 \tau_h u)|^p\,dx\right)^{\frac{1}{p}}\cr\cr
	&=& c\left[\left(\int_{B_t}(1+|Du(x)|)^{\frac{p}{p-1}\frac{p+q-2}{2}}\,dx\right)^{\frac{p-1}{p}}+ \left(\int_{B_t} |b(x)|^{\frac{p}{p-1}}\,dx\right)^{\frac{p-1}{p}}\right]\cr\cr
	&&\quad\cdot\left[\left(\int_{B_t}|\tau_{-h}(\eta D\eta \tau_h u)|^p\,dx\right)^{\frac{1}{p}}+\left(\int_{B_t}|\tau_{-h}(\eta^2 \tau_h u)|^p\,dx\right)^{\frac{1}{p}}\right]\cr\cr
	&\le& c|h|\left[\left(\int_{B_t} (1+|Du(x)|)^{\frac{p}{p-1}\frac{p+q-2}{2}}\,dx\right)^{\frac{p-1}{p}}+ \left(\int_{B_R} |b(x)|^{\frac{p}{p-1}}\,dx\right)^{\frac{p-1}{p}}\right]\cr\cr
	&&\qquad\cdot\left[\left(\int_{B_t}|D(\eta D\eta \tau_h u)|^p\,dx\right)^{\frac{1}{p}}+\left(\int_{B_t}|D(\eta^2 \tau_h u)|^p\,dx\right)^{\frac{1}{p}}\right],
	\end{eqnarray*}
	where, in the last inequality, we used the first statement of Lemma \ref{le1}. Performing the calculations in the last two integrals, using the properties of $\eta$  and assuming without loss of generality that $0<t-s<R<1$,
	we get
	\begin{eqnarray}\label{I_2-I_3}
	I_2+I_3
	 &\le& c|h|\left[\left(\int_{B_t} (1+|Du(x)|)^{\frac{p}{p-1}\frac{p+q-2}{2}}\,dx\right)^{\frac{p-1}{p}}+||b||_{L^{\frac{p}{p-1}}(B_R)}\right]\cr\cr
	 &&\qquad \cdot\left[\frac{c}{(t-s)^2}\left(\int_{B_t}|\tau_h u|^p\,dx\right)^{\frac{1}{p}}+\frac{c}{(t-s)}\left(\int_{B_t} \eta^p|\tau_h Du|^p\,dx\right)^{\frac{1}{p}}\right]\cr\cr
	 &=& \frac{c|h|^2}{(t-s)^2}\left[\left(\int_{B_t} (1+|Du(x)|)^{\frac{p}{p-1}\frac{p+q-2}{2}}\,dx\right)^{\frac{p-1}{p}}+ ||b||_{L^{\frac{p}{p-1}}(B_R)}\right]\cr\cr
	 &&\qquad\qquad \cdot\left(\int_{B_t}|D u|^p\,dx\right)^{\frac{1}{p}}\cr\cr
	 &&+\frac{c|h|}{t-s}\left[\left(\int_{B_t} (1+|Du(x)|)^{\frac{p}{p-1}\frac{p+q-2}{2}}\,dx\right)^{\frac{p-1}{p}}+ ||b||_{L^{\frac{p}{p-1}}(B_R)}\right]\cr\cr
	 &&\qquad\qquad\cdot\left(\int_{B_t} \eta^p|\tau_hD u|^p\,dx\right)^{\frac{1}{p}}.
	\end{eqnarray}
	Inserting \eqref{I1} and \eqref{I_2-I_3} in \eqref{stima} and using the notation $p'=\frac{p}{p-1}$, we obtain
\begin{eqnarray}\label{stima2}
	&&\int_{B_t} \eta^2(\mu^2+|Du(x)|^2+ |Du(x+h)|^2)^{\frac{p-2}{2}}|\tau_hDu|^2\,dx\cr\cr
	&\le& c|h|^2||k||_{L^{r}(B_R)}^2\left(\int_{B_t} \eta^2 (1+|Du(x)|)^{\frac{rq}{r-2}}\,dx\right)^{\frac{r-2}{r}}\cr\cr
	&+&\frac{c|h|^2}{(t-s)^2}\left[\left(\int_{B_t} (1+|Du(x)|)^{\frac{p'(p+q-2)}{2}}\,dx\right)^{\frac{1}{p'}}\!\!+\! ||b||_{L^{p'}(B_R)}\right]\left(\int_{B_t}|D u|^p\,dx\right)^{\frac{1}{p}}\cr\cr
	 &+&\frac{c|h|}{t-s}\left[\left(\int_{B_t} (1+|Du(x)|)^{\frac{p'(p+q-2)}{2}}\,dx\right)^{\frac{1}{p'}}\!\!\!\!+\! ||b||_{L^{p'}(B_R)}\right]\!\!\left(\int_{B_t} \eta^p|\tau_hD u|^p\,dx\right)^{\frac{1}{p}}
	\end{eqnarray}
	
	By the assumption $1<p<2$, we have that
	\begin{eqnarray}\label{stima3}
	&&\left(\int_{B_t}\eta^p|\tau_hD u|^p\,dx\right)^{\frac{1}{p}}\cr\cr
	&=&\left(\int_{B_t}\eta^p|\tau_hD u|^p(\mu^2+|Du(x)|^2+ |Du(x+h)|^2)^{\frac{p(p-2)}{4}}(\mu^2+|Du(x)|^2+ |Du(x+h)|^2)^{\frac{p(2-p)}{4}}\,dx\right)^{\frac{1}{p}}\cr\cr
	&\le &\left(\int_{B_t}\eta^2 |\tau_hD u|^2(\mu^2+|Du(x)|^2+ |Du(x+h)|^2)^{\frac{p-2}{2}}\right)^{\frac{1}{2}}\left(\int_{B_R}(1+|Du|^2)^{\frac{p}{2}}\,dx\right)^{\frac{2-p}{2p}},
	\end{eqnarray}
	and, using \eqref{stima3} to estimate the last integral in the right hand side of \eqref{stima2}, we arrive at
	\begin{eqnarray}\label{sti}
	&&\nu\int_{B_t} \eta^2(\mu^2+|Du(x)|^2+ |Du(x+h)|^2)^{\frac{p-2}{2}}|\tau_hDu|^2\,dx\cr\cr
	&\le& c|h|^2||k||_{L^{r}(B_R)}^2\left(\int_{B_t} \eta^2 (1+|Du(x)|)^{\frac{qr}{r-2}}\,dx\right)^{\frac{r-2}{r}}\cr\cr
	&&+\frac{c|h|^2}{(t-s)^2}\left[\left(\int_{B_t} (1+|Du(x)|)^{\frac{p'(p+q-2)}{2}}\,dx\right)^{\frac{1}{p'}}+ ||b||_{L^{p'}(B_R)}\right]\cr\cr
	 &&\qquad\qquad \cdot\left(\int_{B_t}(1+|D u|)^p\,dx\right)^{\frac{1}{p}}\cr\cr
	 &&+\frac{c|h|}{(t-s)}\left[\left(\int_{B_t} (1+|Du(x)|)^{\frac{p'(p+q-2)}{2}}\,dx\right)^{\frac{p-1}{p}}+ ||b||_{L^{p'}(B_R)}\right]\cr\cr
	 &&\qquad\left(\int_{B_t}\eta^2 |\tau_hD u|^2(\mu^2+|Du(x)|^2+ |Du(x+h)|^2)^{\frac{p-2}{2}}\right)^{\frac{1}{2}}\cr\cr
	 &&\qquad\cdot\left(\int_{B_R}(1+|Du|^2)^{\frac{p}{2}}\,dx\right)^{\frac{2-p}{2p}}
	 \end{eqnarray}
By the use of Young's inequality, we get	
\begin{eqnarray}
	&&\nu\int_{B_t} \eta^2(\mu^2+|Du(x)|^2+ |Du(x+h)|^2)^{\frac{p-2}{2}}|\tau_hDu|^2\,dx\cr\cr
	&\le& c|h|^2||k||^2_{L^{r}(B_R)}\left(\int_{B_t} \eta^2 (1+|Du(x)|)^{\frac{qr}{r-2}}\,dx\right)^{\frac{r-2}{r}}\cr\cr
	&&+\frac{c|h|^2}{(t-s)^2}\left[\left(\int_{B_t} (1+|Du(x)|)^{\frac{p'(p+q-2)}{2}}\,dx\right)^{\frac{1}{p'}}+ ||b||_{L^{p'}(B_R)}\right]\left(\int_{B_R}(1+|D u|)^p\,dx\right)^{\frac{1}{p}}\cr\cr
	 &&+\frac{c|h|^2}{(t-s)^2}\left[\left(\int_{B_t} (1+|Du(x)|)^{\frac{p'(p+q-2)}{2}}\,dx\right)^{\frac{1}{p'}}+ ||b||_{L^{p'}(B_R)}\right]^2\left(\int_{B_R}(1+|Du|^2)^{\frac{p}{2}}\,dx\right)^{\frac{2-p}{p}}\cr\cr
	 &&\qquad\quad +\frac{\nu}{2}\left(\int_{B_t}\eta^2 |\tau_hD u|^2(\mu^2+|Du(x)|^2+ |Du(x+h)|^2)^{\frac{p-2}{2}}\right)
	 \end{eqnarray} 
	Reabsorbing the last integral in the right hand side, we get
	\begin{eqnarray*}
	&&\frac{\nu}{2}\int_{B_t} \eta^2(\mu^2+|Du(x)|^2+ |Du(x+h)|^2)^{\frac{p-2}{2}}|\tau_hDu|^2\,dx\cr\cr
	&\le& c|h|^2||k||_{L^{r}(B_R)}^2\left(\int_{B_t} \eta^2 (1+|Du(x)|)^{\frac{rq}{r-2}}\,dx\right)^{\frac{r-2}{r}}\cr\cr
	&&+\frac{c|h|^2}{(t-s)^2}\left[\left(\int_{B_t} (1+|Du(x)|)^{\frac{p'(p+q-2)}{2}}\,dx\right)^{\frac{1}{p'}}+ ||b||_{L^{p'}(B_R)}\right]\left(\int_{B_R}(1+|D u|)^p\,dx\right)^{\frac{1}{p}}\cr\cr
	 &&+\frac{c|h|^2}{(t-s)^2}\left[\left(\int_{B_t} (1+|Du(x)|)^{\frac{p'(p+q-2)}{2}}\,dx\right)^{\frac{1}{p'}}+ ||b||_{L^{p'}(B_R)}\right]^2\left(\int_{B_R}(1+|Du|)^{p}\,dx\right)^{\frac{2-p}{p}}.
	 \end{eqnarray*}
	 By Young's inequality in the second term in the right hand side of previous estimate, we have
	 \begin{eqnarray*}
	&&\frac{\nu}{2}\int_{B_t} \eta^2(\mu^2+|Du(x)|^2+ |Du(x+h)|^2)^{\frac{p-2}{2}}|\tau_hDu|^2\,dx\cr\cr
	&\le& c|h|^2||k||_{L^{r}(B_R)}^2\left(\int_{B_t} \eta^2 (1+|Du(x)|)^{\frac{rq}{r-2}}\,dx\right)^{\frac{r-2}{r}}+\frac{c|h|^2}{(t-s)^2}\left(\int_{B_R}(1+|D u|)^p\,dx\right)^{\frac{2}{p}}\cr\cr
	 &&+\frac{c|h|^2}{(t-s)^2}\left[\left(\int_{B_t} (1+|Du(x)|)^{\frac{p'(p+q-2)}{2}}\,dx\right)^{\frac{1}{p'}}+ ||b||_{L^{p'}(B_R)}\right]^2\cr\cr
	&&\qquad\qquad\cdot\left[1+\left(\int_{B_R}(1+|Du|^2)^{\frac{p}{2}}\,dx\right)^{\frac{2-p}{p}}
\right].
	 \end{eqnarray*}
	 So, by the elementary inequality $(a+b)^2\le c(a^2+b^2)$, we get
	 \begin{eqnarray*}
	&&\frac{\nu}{2}\int_{B_t} \eta^2(\mu^2+|Du(x)|^2+ |Du(x+h)|^2)^{\frac{p-2}{2}}|\tau_hDu|^2\,dx\cr\cr
	&\le& c|h|^2||k||_{L^{r}(B_R)}^2\left(\int_{B_t} \eta^2 (1+|Du(x)|)^{\frac{rq}{r-2}}\,dx\right)^{\frac{r-2}{r}}\cr\cr
	&&+\frac{c|h|^2}{(t-s)^2}\left(\int_{B_t} (1+|Du(x)|)^{\frac{p'(p+q-2)}{2}}\,dx\right)^{\frac{2}{p'}}\cr\cr
	&&\qquad\qquad\cdot\left[1+\left(\int_{B_R}(1+|Du|^2)^{\frac{p}{2}}\,dx\right)^{\frac{2-p}{p}}
\right]\cr\cr
	 &&+\frac{c|h|^2}{(t-s)^2}||b||_{L^{p'}(B_R)}^2\left[1+\left(\int_{B_R}(1+|Du|)^{p}\,dx\right)^{\frac{2-p}{p}}
\right]\cr\cr
&&+\frac{c|h|^2}{(t-s)^2}\left(\int_{B_R}(1+|D u|)^p\,dx\right)^{\frac{2}{p}}
	 \end{eqnarray*}

	Dividing both sides of previous estimate by $|h|^2$ and using Lemma \ref{Vi}, we have
	\begin{eqnarray}\label{stimadiff0}
	&&\int_{B_t} \eta^2\frac{|\tau_hV_p(Du)|^2}{|h|^2}\,dx\cr\cr
	&\le& c||k||_{L^{p'}(B_R)}^2\left(\int_{B_t} \eta^2 (1+|Du(x)|)^{\frac{rq}{r-2}}\,dx\right)^{\frac{r-2}{r}}\cr\cr
	&&+\frac{c}{(t-s)^2}\left(\int_{B_t} (1+|Du(x)|)^{\frac{p'(p+q-2)}{2}}\,dx\right)^{\frac{2}{p'}}\cr\cr
	&&\qquad\qquad\cdot\left[1+\left(\int_{B_R}(1+|Du|^2)^{\frac{p}{2}}\,dx\right)^{\frac{2-p}{p}}
\right]\cr\cr
	 &&+\frac{c}{(t-s)^2}||b||_{L^{p'}(B_R)}^2\left[1+\left(\int_{B_R}(1+|Du|)^{p}\,dx\right)^{\frac{2-p}{p}}
\right]\cr\cr
&&+\frac{c}{(t-s)^2}\left(\int_{B_R}(1+|D u|)^p\,dx\right)^{\frac{2}{p}}
	 \end{eqnarray}
Note the right hand side of \eqref{stimadiff0} is finite by virtue of the a-priori assumption $Du\in L^{\frac{2^*}{2}p}_{\mathrm{loc}}(\Omega)$ and the bound at \eqref{gap}. Indeed, the bound on the gap at \eqref{gap}, implies that
 \begin{eqnarray}\label{bound}
 \begin{cases}p<\frac{qr}{r-2}<\frac{2^*}{2}p
 \\
 \text{and}
 \\
 p<\frac{p}{p-1}\frac{p+q-2}{2}<\frac{2^*}{2}p
 \end{cases}	
 \end{eqnarray}
 {Indeed for $n>2$ we have $\frac{2^*}{2}=\frac{n}{n-2}$ and
 $$ \frac{qr}{r-2}<\frac{np}{n-2}\,\,\Longleftrightarrow\,\, \frac{q}{p}<\frac{r-2}{r}\frac{n}{n-2}=1+\frac{2}{n-2}-\frac{2n}{r(n-2)}$$
 and
 $$\frac{2}{n-2}-\frac{2n}{r(n-2)}=\frac{2(r-n)}{r(n-2)}>\frac{1}{n}-\frac{1}{r}\,\,\Longleftrightarrow\,\, \frac{2}{n-2}>\frac{1}{n}.$$
 Moreover
 $$\frac{p}{p-1}\frac{p+q-2}{2}<\frac{np}{n-2} \,\,\Longleftrightarrow\,\, p+q-2<\frac{2n(p-1)}{n-2}$$
 i.e.
 $$q<\frac{2n(p-1)}{n-2}-p+2=p+\frac{2n(p-1)}{n-2}-2(p-1)=p+2(p-1)\left[\frac{n}{n-2}-1\right] $$
 and so
 $$\frac{q}{p}<1+\frac{4n(p-1)}{p(n-2)} $$
 On the other hand for $n=2$, since $\frac{2^*}{2}$ is any exponent larger than $2$ inequalities \eqref{bound} are trivially satisfied choosing $2^*$ sufficiently large.}

	Therefore, estimate \eqref{stimadiff0}, by the  use of  Lemma \ref{lep} with $p=2$, yields
	 $$V_p(Du)\in W^{1,2}(B_s)\cap L^{\frac{2^*}{2}}(B_s).$$
	 Moreover by  \eqref{stinorma} in Lemma \ref{rem} and \eqref{frac}, we get
	\begin{eqnarray}\label{stima4}
	&&\left(\int_{B_s} |Du|^{\frac{2^*}{2}p}\,dx\right)^{\frac{2}{2^*}}\le c(n)\left(\int_{B_s}\left( 1+|V_p(Du)|^{2^*}\right)\,dx\right)^{\frac{2}{2^*}}\cr\cr
	&&\quad \le c||k||_{L^{r}(B_R)}^2\left(\int_{B_t} \eta^2 (1+|Du(x)|)^{\frac{rq}{r-2}}\,dx\right)^{\frac{r-2}{r}}\cr\cr
	&&+\frac{c}{(t-s)^2}\left(\int_{B_t} (1+|Du(x)|)^{\frac{p'(p+q-2)}{2}}\,dx\right)^{\frac{2}{p'}}\cr\cr
	&&\qquad\qquad\cdot\left[1+\left(\int_{B_R}(1+|Du|^2)^{\frac{p}{2}}\,dx\right)^{\frac{2-p}{p}}
\right]\cr\cr
	 &&+\frac{c}{(t-s)^2}||b||_{L^{p'}(B_R)}^2\left[1+\left(\int_{B_R}(1+|Du|)^{p}\,dx\right)^{\frac{2-p}{p}}
\right]\cr\cr
&&+\frac{c}{(t-s)^2}\left(\int_{B_R}(1+|D u|)^p\,dx\right)^{\frac{2}{p}}
\end{eqnarray}
Our next aim is to establish an estimate independent on the a-priori assumption $Du\in L^{\frac{2^*}{2}p}_{\mathrm{loc}}(\Omega).$
Setting 
$$\mathbb{J}_1=\left(\int_{B_t}  (1+|Du(x)|)^{\frac{qr}{r-2}}\,dx\right)^{\frac{r-2}{r}}$$ 
 and
 $$\mathbb{J}_2=\left(\int_{B_t} (1+|Du(x)|)^{\frac{p'(p+q-2)}{2}}\,dx\right)^{\frac{2}{p'}}$$
we can rewrite \eqref{stima4} as follows

\begin{eqnarray}\label{stima4b}
&&\left(\int_{B_s} |Du|^{\frac{2^*}{2}p}\,dx\right)^{\frac{2}{2^*}}\le c||k||_{L^{r}(B_R)}^2\,\,\mathbb{J}_1\cr\cr
&&\quad+\frac{c\,\mathbb{J}_2}{(t-s)^2}\left[1+\left(\int_{B_R}(1+|Du|)^{p}\,dx\right)^{\frac{2-p}{p}}
\right]\cr\cr
	 &&\qquad+\frac{c}{(t-s)^2}||b||_{L^{p'}(B_R)}^2\left[1+\left(\int_{B_R}(1+|Du|)^{p}\,dx\right)^{\frac{2-p}{p}}
\right]\cr\cr
&&\quad\qquad+\frac{c}{(t-s)^2}\left(\int_{B_R}(1+|D u|)^p\,dx\right)^{\frac{2}{p}}.
	 \end{eqnarray} 
 \\
 By the inequalities in \eqref{bound}, there exist $\vartheta_1\in (0,1)$ and $\vartheta_2\in (0,1)$ such that
 $$1=\frac{\vartheta_1\frac{qr}{r-2}}{\frac{2^*p}{2}}+\frac{(1-\vartheta_1)\frac{qr}{r-2}}{p}$$
 and 
 $$1=\frac{\vartheta_2\frac{p}{p-1}\frac{p+q-2}{2}}{\frac{2^*p}{2}}+\frac{(1-\vartheta_2)\frac{p}{p-1}\frac{p+q-2}{2}}{p}$$
 It is easy to check that
 $$\vartheta_1=\frac{2^*-2}{2}\frac{(q-p)r+2p}{qr}$$
 and 
 $$\vartheta_2=\frac{2^*-2}{2}\frac{q-p}{p+q-2}$$
 Therefore, the interpolation  inequality implies
\begin{eqnarray}\label{J1}
 	\mathbb{J}_1 &\le & \left(\int_{B_t}  (1+|Du(x)|)^{\frac{2^*p}{2}}\,dx\right)^{\frac{\vartheta_1\frac{qr}{r-2}}{\frac{2^*p}{2}}\frac{r-2}{r}}\cr\cr
 	&&\qquad \cdot\left(\int_{B_t}  (1+|Du(x)|)^{p}\,dx\right)^{\frac{(1-\vartheta_1)\frac{qr}{r-2}}{p}\frac{r-2}{r}}\cr\cr
 	&= &\left(\int_{B_t}  (1+|Du(x)|)^{\frac{2^*p}{2}}\,dx\right)^{\frac{2^*-2}{2}\frac{(q-p)r+2p}{pr}\frac{2}{2^*}}\cr\cr
 	&&\qquad \cdot\left(\int_{B_t}  (1+|Du(x)|)^{p}\,dx\right)^{(1-\vartheta_1)\frac{q}{p}}
 \end{eqnarray}
and 
 \begin{eqnarray}\label{J2}
 	\mathbb{J}_2 &\le & \left(\int_{B_t}  (1+|Du(x)|)^{\frac{2^*p}{2}}\,dx\right)^{\frac{\vartheta_2\frac{p'(q+p-2)}{2}}{\frac{2^*p}{2}}\frac{2}{p'}}\cr\cr
 	&&\qquad \cdot\left(\int_{B_t}  (1+|Du(x)|)^{p}\,dx\right)^{\frac{(1-\vartheta_2)\frac{p'(q+p-2)}{2}}{p}\frac{2}{p'}}\cr\cr
 	&= &\left(\int_{B_t}  (1+|Du(x)|)^{\frac{2^*p}{2}}\,dx\right)^{\frac{2^*-2}{2}\frac{q-p}{p}\frac{2}{2^*}}\cr\cr
 	&&\qquad \cdot\left(\int_{B_t}  (1+|Du(x)|)^{p}\,dx\right)^{\frac{(1-\vartheta_2)(q+p-2)}{p}},
 \end{eqnarray}
 where we used the  expressions of $\vartheta_1,\vartheta_2$.
 Using   estimates   \eqref{J1} and \eqref{J2}  in \eqref{stima4b} we obtain
\begin{eqnarray}\label{stima5}
	 &&\left(\int_{B_s}|D u|^{\frac{2^*p}{2}}\,dx\right)^{\frac{2}{2^*}}\cr\cr
	  &\le& c||k||_{L^r(B_R)}^2\left(\int_{B_t}  (1+|Du(x)|)^{\frac{2^*p}{2}}\,dx\right)^{\frac{2^*-2}{2}\frac{(q-p)r+2p}{pr}\frac{2}{2^*}}\left(\int_{B_t}  (1+|Du(x)|)^{p}\,dx\right)^{\frac{(1-\vartheta_1)q}{p}}\cr\cr
	 &&+\frac{c}{(t-s)^2}\left(\int_{B_t}  (1+|Du(x)|)^{\frac{2^*p}{2}}\,dx\right)^{\frac{2^*-2}{2}\frac{q-p}{p}\frac{2}{2^*}}\left(\int_{B_t}  (1+|Du(x)|)^{p}\,dx\right)^{\frac{(1-\vartheta_2)(q-p+2)}{p}}\cr\cr
 	&&\qquad \cdot\left[1+\left(\int_{B_R}(1+|Du|)^{p}\,dx\right)^{\frac{2-p}{p}}
\right]\cr\cr
	&&+\frac{c}{(t-s)^2}||b||_{L^{p'}(B_R)}^2\left[1+\left(\int_{B_R}(1+|Du|)^{p}\,dx\right)^{\frac{2-p}{p}}
\right]\cr\cr
&&+\frac{c}{(t-s)^2}\left(\int_{B_R}(1+|D u|)^p\,dx\right)^{\frac{2}{p}}. \end{eqnarray}

	 Let us note that
	 $$\frac{(2^*-2)}{2}\frac{(q-p)r+2p}{rp}<1 \quad\Longleftrightarrow\quad\frac{q}{p}<1+2\left(\frac{1}{n}-\frac{1}{r}\right) $$
	which is precisely \eqref{gap}, and
	$$ \frac{(2^*-2)}{2}\frac{q-p}{p} <1 \quad\Longleftrightarrow\quad \frac{q}{p}<1+\frac{2}{n}$$
	which is less restrictive than \eqref{gap}. Therefore
	by Young's inequality with exponents 
	$$\frac{2rp}{(2^*-2)[r(q-p)+2p]},\qquad \frac{2rp}{2rp-(2^*-2)[r(q-p)+2p]}$$ 
	in the first term of the right hand side of \eqref{stima5} and with exponents  
	$$\frac{2p}{(2^*-2)(q-p)},\qquad \frac{2p}{2p-(2^*-2)(q-p)}$$ 
	in the second term to get
	\begin{eqnarray}\label{stima6}
	 &&\left(\int_{B_s}|D u|^{\frac{2^*p}{2}}\,dx\right)^{\frac{2}{2^*}}\le \frac{1}{2}\left(\int_{B_t}  (1+|Du(x)|)^{\frac{2^*p}{2}}\,dx\right)^{\frac{2}{2^*}}\cr\cr
	 &&\quad+ c||k||_{L^r(B_R)}^{\gamma_1}\left(\int_{B_R}  (1+|Du(x)|)^{p}\,dx\right)^{\gamma_2}\cr\cr
	 &&+\frac{c}{(t-s)^{2\gamma_0}}\left(\int_{B_R}  (1+|Du(x)|)^{p}\,dx\right)^{\gamma_3}\cr\cr
	 &&+\frac{c}{(t-s)^2}||b||_{L^{p'}(B_R)}^2\left[1+\left(\int_{B_R}(1+|Du|)^{p}\,dx\right)^{\frac{2-p}{p}}
\right]\cr\cr
&&+\frac{c}{(t-s)^2}\left(\int_{B_R}(1+|D u|)^p\,dx\right)^{\frac{2}{p}},
	 \end{eqnarray}
	 where $\gamma_i>1$ only depend on $p,q,r,n$, for $i=0,1,2,3$.
The iteration Lemma \ref{iter} implies
\begin{equation}\label{stima6}
	\left(\int_{B_{\frac{R}{2}}}|D u|^{\frac{2^*p}{2}}\,dx\right)^{\frac{2}{2^*}}\le \frac{c(1+||k||_{L^r(B_R)}+ ||b||_{L^{p'}(B_R)})^\gamma}{R^{2\gamma}}\left(1+\int_{B_R}  (1+|Du|)^{p}\,dx\right)^{\gamma}	 \end{equation}
	 for every ball $B_R\Subset\Omega$ for a constant  $ c=c(n,p,q,L,\nu)$ and with a positive exponent $\gamma=\gamma(n,p,q,r)$. Finally the integrals $\mathbb{J}_1$ and $\mathbb{J}_2$ can be estimated as follows
	 \begin{eqnarray}\label{J_1}
	\mathbb{J}_1\le c\left(\int_{B_R}  (1+|Du(x)|)^{p}\,dx\right)^{\gamma_4}	
	 \end{eqnarray} 
 and
 \begin{eqnarray}\label{J_2}
 	\mathbb{J}_2\le c\left(\int_{B_R}  (1+|Du(x)|)^{p}\,dx\right)^{\gamma_5},
 \end{eqnarray}
with  $ c=c(n,R,p,q,L,\nu, ||k||_{L^r(B_R)}, ||b||_{L^{p'}(B_R)})$ and  positive exponents $\gamma_4=\gamma_4(n,p,q,r)$ and $\gamma_5=\gamma_5(n,p,q,r)$. Therefore, inserting these estimates in \eqref{stimadiff0}, we obtain 
\begin{eqnarray}
	 \int_{B_{\frac{R}{2}}}\frac{|\tau_hV_p(D u)|^2}{|h|^2}\,dx
	 \le c(1+||k||_{L^r(B_R)}+ ||b||_{L^{p'}(B_R)})^{\vartheta}\left(1+\int_{B_R}  (1+|Du(x)|)^{p}\,dx\right)^{\vartheta}	 \end{eqnarray}
	 where, again, $ c=c(n,R,p,q,L,\nu)$ and  $\vartheta=\vartheta(n,p,q,r)>0$.
	 Denoting by
	 \begin{eqnarray*}
	 	M^2&=&c(1+||k||_{L^r(B_R)}+ ||b||_{L^{p'}(B_R)})^{\vartheta}\left(\int_{B_R}  (1+|Du(x)|)^{p}\,dx\right)^{\vartheta}	 \end{eqnarray*} we conclude, by the use of Lemma \ref{lep}, that
	 $$V_p(Du)\in W^{1,2}_{\mathrm{loc}}(\Omega)$$
	 and therefore, by Lemma \ref{VD},
	 $$D^2u\in L^p_{\mathrm{loc}}(\Omega).$$
	 Moreover  estimates \eqref{dersec2} and \eqref{dersec1} hold true.
	\end{proof}
	
	As a consequence of previous theorem, we deduce the following
\begin{theorem}\label{main2}
Let $u\in W^{1,\frac{2^*p}{2}}_{\mathrm{loc}}(\Omega)$ be a solution to \eqref{equa} under the assumptions \eqref{(A2)}--\eqref{(A3)} with $1<p<q$ such that
\begin{equation}\label{gap12}
	\frac{q}{p}<1+\min\left\{\frac{1}{n}-\frac{1}{r},\, \frac{2(p-1)}{p(n-2)}\right\}\qquad \text{if}\,\, n>2
\end{equation}
or
\begin{equation}\label{gap1n=2}
	\frac{q}{p}<1+\frac{1}{n}-\frac{1}{r},\qquad\qquad\qquad\qquad\qquad \text{if}\,\, n=2.
\end{equation}
Then    estimates \eqref{dersec2} and \eqref{dersec1} hold true.
\end{theorem}
\begin{proof}
The proof follows as that of Theorem 5.1 in \cite{mar91}. 	Let $t=2q-p$ then \eqref{(A2)}-- \eqref{(A3)} are nothing else but  \eqref{(A2b)}-- \eqref{(A3b)} with $t$ in place of $q$. Moreover 
$$\frac{t}{p}=\frac{2q}{p}-1$$
and so the validity of \eqref{gap12} and \eqref{gap1n=2} implies the validity of \eqref{gap} and \eqref{gapn=2}, respectively.
\end{proof}

\vskip1cm

\section{Proof of Theorem \ref{main1}}
	 Let $\phi\in C_c^\infty(B_1(0))$, $\phi\ge 0$, be such that $\int_{B_1(0)}\phi\,dx=1$ and let $\phi_\varepsilon$ be the corresponding family of mollifiers. Let us set
	 $$\tilde{\mathcal{A}}_\varepsilon(x,\xi)=\int_{B_1(0)}\phi(\omega)\mathcal{A}(x+\varepsilon\omega,\xi)\,d\omega,$$
	 
	 \medskip
	 
	  $$\mathcal{A}_\varepsilon(x,\xi)=\tilde{\mathcal{A}}_\varepsilon(x,\xi)+\varepsilon(1+|\xi|^2)^{\frac{q-2}{2}}\xi$$
	  and
	  $$b_\varepsilon=b\star \phi_\varepsilon.$$
	  Let $u_0\in W^{1,\frac{p(q-1)}{p-1}}(\Omega)$, fix a ball $B_R\Subset \Omega$ and let $v_\varepsilon\in u_0+ W^{1,q}_0(\Omega)$ be the solution to the Dirichlet problem
	  \begin{equation}\label{equaapp}
	  	\begin{cases}\mathrm{div} \mathcal{A}_\varepsilon(x,Dv_\varepsilon)=b_\varepsilon \qquad\qquad \mathrm{in}\,\, \Omega \cr\cr
	  	v_\varepsilon=u_0\quad\qquad\qquad\qquad\qquad \mathrm{on}\,\,\partial \Omega	
	  	\end{cases}
\end{equation}
	  		
%
	   Note that
	   
	   \medskip
	   
	  \begin{equation}\label{B1}
	| \mathcal{A}_\varepsilon(x,\xi)-\mathcal{A}_\varepsilon(x,\eta)|\le |\xi-\eta|\left[L(\mu^2+1+|\xi|^2+|\eta|^2)^{\frac{p-2}{2}}+c_q\varepsilon(1+|\xi|^2+|\eta|^2)^{\frac{q-2}{2}}\right]
\end{equation}

\medskip

\begin{equation}\label{B2}
	\langle \mathcal{A}_\varepsilon(x,\xi)-\mathcal{A}_\varepsilon(x,\eta),\xi-\eta\rangle\ge \nu|\xi-\eta|^2(\mu^2+|\xi|^2+|\eta|^2)^{\frac{p-2}{2}}+\varepsilon c_q|\xi-\eta|^2(\mu^2+|\xi|^2+|\eta|^2)^{\frac{q-2}{2}}
\end{equation}

\medskip

\begin{equation}\label{B3}
	| \mathcal{A}_\varepsilon(x,\xi)-\mathcal{A}_\varepsilon(y,\xi)|\le |x-y|(k_\varepsilon(x)+k_\varepsilon(y))(1+|\xi|^2)^{\frac{q-1}{2}}
\end{equation}
	 where	 $$k_\varepsilon=k\star \phi_\varepsilon.$$
Note that assumption \eqref{(A1')} implies that
\begin{equation}\label{(B1')}
	| \mathcal{A}_\varepsilon(x,\xi)|\le (L+c_q\varepsilon)|\xi|(1+|\xi|^2)^{\frac{q-1}{2}}
\end{equation}	 

	Since the operator $\mathcal{A}_\varepsilon$ satisfies assumptions \eqref{B1}--\eqref{B3} and $\frac{p(q-1)}{p-1}\ge q$, there exists a unique solution $v_\varepsilon\in u_0+W^{1,q}(\Omega)$ and by the regularity results in \cite{Tolksdorf}, we have that $v_\varepsilon\in W^{2,q}_{\mathrm{loc}}(\Omega)$ and $V_q(Dv_\varepsilon)=(\mu^2+|Dv_\varepsilon|^2)^{\frac{q-2}{4}}Dv_\varepsilon\in W^{1,2}_{\mathrm{loc}}(\Omega)$.
	\\
	Therefore, by Lemma \ref{rem},  $Dv_\varepsilon\in L^{\frac{2^*q}{2}}_{\mathrm{loc}}(\Omega)\subset L^{\frac{2^*p}{2}}_{\mathrm{loc}}(\Omega)$ and so
	 we are legitimate to apply estimates \eqref{dersec2} and \eqref{dersec1} to each $v_\varepsilon$ to obtain that
	\begin{eqnarray}\label{dersec1b}
\int_{B_{\frac{r}{2}}}|D^2 v_\varepsilon|^p\,dx\le c(1+||k_\varepsilon||_{L^r(B_R)}+ ||b_\varepsilon||_{L^{p'}(B_R)})^{\gamma}\left(1+\int_{B_r}  (1+|Dv_\varepsilon(x)|)^{p}\,dx\right)^{\gamma}
\end{eqnarray}
and
\begin{eqnarray}\label{dersec2b}
	\int_{B_{\frac{r}{2}}}|D(V_p(D v_\varepsilon))|^2\,dx
	\le  c(1+||k_\varepsilon||_{L^r(B_R)}+ ||b_\varepsilon||_{L^{p'}(B_R)})^{\vartheta}\left(1+\int_{B_r}  (1+|Dv_\varepsilon(x)|)^{p}\,dx\right)^{\vartheta},	
\end{eqnarray}
for all $B_r\Subset B_R$. Moreover by assumption \eqref{B2} and since $v_\varepsilon$ solves problem \eqref{equaapp} we get
\begin{eqnarray}\label{stimalp0}
	&&\nu\int_{\Omega}|Dv_\varepsilon-Du_0|^2(\mu^2+|Dv_\varepsilon|^2+|Du_0|^2)^{\frac{p-2}{2}}\,dx\le \int_{\Omega}\langle \mathcal{A}_\varepsilon(x,Dv_\varepsilon)-\mathcal{A}_\varepsilon(x,Du_0), Dv_\varepsilon-Du_0\rangle\, dx\cr\cr
	&=&\int_{\Omega}\langle \mathcal{A}_\varepsilon(x,Dv_\varepsilon), Dv_\varepsilon-Du_0\rangle\, dx-\int_{\Omega}\langle \mathcal{A}_\varepsilon(x,Du_0), Dv_\varepsilon-Du_0\rangle\, dx\cr\cr
	&=&\int_{\Omega}\langle b_\varepsilon, Dv_\varepsilon-Du_0\rangle\, dx-\int_{\Omega}\langle \mathcal{A}_\varepsilon(x,Du_0), Dv_\varepsilon-Du_0\rangle\, dx\cr\cr
	&\le& \int_{\Omega}| b_\varepsilon|| Dv_\varepsilon-Du_0|\, dx+\int_{\Omega}| \mathcal{A}_\varepsilon(x,Du_0)||Dv_\varepsilon-Du_0|\, dx.\end{eqnarray}
	As long as $1<p<2$, we can use Young's inequality with exponents $\frac{2}{p}$ and $\frac{2}{2-p}$ as follows
	\begin{eqnarray}\label{stimalp1}
	&&\int_{\Omega}|Dv_\varepsilon-Du_0|^p\,dx \cr\cr
	&=&\int_{\Omega}|Dv_\varepsilon-Du_0|^p(\mu^2+|Dv_\varepsilon|^2+|Du_0|^2)^{\frac{p-2}{2}\frac{p}{2}}(\mu^2+|Dv_\varepsilon|^2+|Du_0|^2)^{\frac{2-p}{2}\frac{p}{2}}\,dx\cr\cr
	&\le& \nu\int_{\Omega}|Dv_\varepsilon-Du_0|^2(\mu^2+|Dv_\varepsilon|^2+|Du_0|^2)^{\frac{p-2}{2}}\,dx	\cr\cr
	&&\qquad +c(\nu)\int_{\Omega}(\mu^2+|Dv_\varepsilon|^2+|Du_0|^2)^{\frac{p}{2}}\,dx
	\end{eqnarray}
	Using \eqref{stimalp0} to estimate the first term in the right hand side of \eqref{stimalp1}, we get
	\begin{eqnarray}\label{stimalp2}
	&&\int_{\Omega}|Dv_\varepsilon-Du_0|^p\,dx\cr\cr
	&\le&\frac{1}{2}\int_{\Omega}|Dv_\varepsilon-Du_0|^p\,dx+ c\int_{\Omega}| b_\varepsilon|^{\frac{p}{p-1}}+c\int_{\Omega} |\mathcal{A}_\varepsilon(x,Du_0)|^{\frac{p}{p-1}},
	\end{eqnarray}
	where,  in the last line, we used Young's inequality. Reabsorbing the first integral in the right hand side by the left hand side and using \eqref{(B1')} we have
	\begin{eqnarray*}
	&&\int_{\Omega}|Dv_\varepsilon-Du_0|^p\,dx
	\le   c\int_{\Omega}| b_\varepsilon|^{\frac{p}{p-1}}+c\int_{\Omega} |\mathcal{A}_\varepsilon(x,Du_0)|^{\frac{p}{p-1}}\cr\cr
	&\le& c\int_{\Omega}| b|^{\frac{p}{p-1}}+c\int_{\Omega} |Du_0|^{\frac{p(q-1)}{p-1}}
	\end{eqnarray*}
	and so
	\begin{eqnarray}\label{stimalp3}
	\int_{\Omega}|Dv_\varepsilon|^p\,dx
	\le c\int_{\Omega}| b|^{\frac{p}{p-1}}+c\int_{B_R} |Du_0|^{\frac{p(q-1)}{p-1}}
	\end{eqnarray}
	where we used that $b_\varepsilon$ strongly converges to $b$ in $L^{\frac{p}{p-1}}(\Omega)$. Estimate \eqref{stimalp3} implies that $v_\varepsilon$ is a bounded sequence in $u_0+W^{1,p}_0(\Omega)$ and therefore there exists $v\in u_0+W^{1,p}_0(\Omega)$ such that
	$$v_\varepsilon\rightharpoonup v \qquad \text{weakly in }\,\,u_0+W^{1,p}_0(\Omega).$$ 
	On the other hand, since $b_\varepsilon$ strongly converges to $b$ in $L^{\frac{p}{p-1}}_{\mathrm{loc}}(\Omega)$ and $k_\varepsilon$ strongly converges to $k$ in $L^{r}_{\mathrm{loc}}(\Omega)$, using \eqref{stimalp3} in \eqref{dersec1b} and in \eqref{dersec2b}, we get
	\begin{eqnarray}\label{dersec1c}
\int_{B_{\frac{R}{2}}}|D^2 v_\varepsilon|^p\,dx\le c(1+||k||_{L^r(B_R)}+ ||b||_{L^{p'}(B_R)})^{\gamma}\left(1+\int_{B_R} |Du_0|^{\frac{p(q-1)}{p-1}}\right)^{\gamma}		
\end{eqnarray}
and
\begin{eqnarray}\label{dersec2c}
	\int_{B_{\frac{R}{2}}}|D(V_p(D v_\varepsilon))|^2\,dx\le c(1+||k||_{L^r(B_R)}+ ||b||_{L^{p'}(B_R)})^{\vartheta}\left(1+\int_{B_R} |Du_0|^{\frac{p(q-1)}{p-1}}\right)^{\vartheta},	
\end{eqnarray}
for every balls $ B_R\Subset\Omega$.
Estimate \eqref{dersec1c} implies that $v_\varepsilon$ is a bounded sequence in $W^{2,p}_{\mathrm{loc}}(\Omega)$ and therefore 
	$$v_\varepsilon\rightharpoonup v \qquad \text{weakly in }\,\,W^{2,p}_{\mathrm{loc}}(\Omega)$$
	and so
	 $$v_\varepsilon\to v \qquad \text{strongly in }\,\,W^{1,p}_{\mathrm{loc}}(\Omega).$$
	 Passing to the limit as $\varepsilon\to 0$ in \eqref{dersec1c}, by the lower semicontinuity of the norm, we obtain 
	 \begin{eqnarray}\label{dersec1d}
\int_{B_{\frac{r}{2}}}|D^2 v|^p\,dx\le c(1+||k||_{L^r(B_R)}+ ||b||_{L^{p'}(B_R)})^{\gamma}\left(1+\int_{B_R} |Du_0|^{\frac{p(q-1)}{p-1}}\right)^{\gamma}		
\end{eqnarray}
The continuity of the map $DV_p(\xi)$ implies that the sequence $DV_p(Dv_\varepsilon)$ converges, up to a subsequence, a.e. to $DV_p(Dv)$. Therefore by Fatou's Lemma Passing to the limit as $\varepsilon\to 0$ in \eqref{dersec2c} we also have
\begin{eqnarray}\label{dersec2d}
	\int_{B_{\frac{r}{2}}}|D(V_p(D v))|^2\,dx\le c(1+||k||_{L^r(B_R)}+ ||b||_{L^{p'}(B_R)})^{\vartheta}\left(1+\int_{B_R} |Du_0|^{\frac{p(q-1)}{p-1}}\right)^{\vartheta}.	
\end{eqnarray}
It remains to prove that $v$ satisfies \eqref{equaapp}. To this aim, we observe that
\begin{eqnarray*}
	&&\int_{\Omega}\langle\mathcal{A}(x,Dv),D\varphi\rangle\,dx-\int_{\Omega}\langle b,D\varphi\rangle\,dx\cr\cr
	&=&\int_{\Omega}\langle\mathcal{A}(x,Dv)-\tilde{\mathcal{A}}_\varepsilon(x,Dv),D\varphi\rangle\,dx+\int_{\Omega}\langle \tilde{\mathcal{A}}_\varepsilon(x,Dv)-\tilde{\mathcal{A}}_\varepsilon(x,Dv_\varepsilon),D\varphi\rangle\,dx\cr\cr
	&&+\int_{\Omega}\langle \tilde{\mathcal{A}}_\varepsilon(x,Dv_\varepsilon),D\varphi\rangle\,dx-\int_{B_R}\langle b,D\varphi\rangle\,dx\cr\cr
	&=&\int_{\Omega}\langle\mathcal{A}(x,Dv)-\tilde{\mathcal{A}}_\varepsilon(x,Dv),D\varphi\rangle\,dx+\int_{B_R}\langle \tilde{\mathcal{A}}_\varepsilon(x,Dv)-\tilde{\mathcal{A}}_\varepsilon(x,Dv_\varepsilon),D\varphi\rangle\,dx\cr\cr
	&&+\int_{\Omega}\langle b_\varepsilon-b,D\varphi\rangle\,dx-\varepsilon \int_{\Omega}\langle (1+|Dv_\varepsilon|^2)^{\frac{q-2}{2}}Dv_\varepsilon,D\varphi\rangle\,dx,
\end{eqnarray*}
for every $\varphi\in C^1_c(\Omega)$.
Therefore, we are left to prove that the right hand side of previous equality vanishes as $\varepsilon\to 0$. 
This will come if we show that
$$\lim_{\varepsilon\to 0}I_1^\varepsilon:=\lim_{\varepsilon\to 0}\left|\int_{\Omega}\langle\mathcal{A}(x,Dv)-\tilde{\mathcal{A}}_\varepsilon(x,Dv),D\varphi\rangle\,dx\right|=0$$
$$\lim_{\varepsilon\to 0}I_2^\varepsilon:=\lim_{\varepsilon\to 0}\left|\int_{\Omega}\langle \tilde{\mathcal{A}}_\varepsilon(x,Dv)-\tilde{\mathcal{A}}_\varepsilon(x,Dv_\varepsilon),D\varphi\rangle\,dx\right|=0$$
$$\lim_{\varepsilon\to 0}I_3^\varepsilon:=\lim_{\varepsilon\to 0}\left|\int_\Omega\langle b_\varepsilon-b,D\varphi\rangle\,dx\right|=0$$
and 
$$\lim_{\varepsilon\to 0}\varepsilon \left|\int_{\Omega}\langle (1+|Dv_\varepsilon|^2)^{\frac{q-2}{2}}Dv_\varepsilon,D\varphi\rangle\,dx\right|=0$$
Last two estimate are obvious since $b_\varepsilon\to b$ strongly in $L^{\frac{p}{p-1}}(\Omega)$ and  since
\begin{eqnarray*}
	&&\lim_{\varepsilon\to 0}\varepsilon \left|\int_{\Omega}\langle (1+|Dv_\varepsilon|^2)^{\frac{q-2}{2}}Dv_\varepsilon,D\varphi\rangle\,dx\right|\cr\cr
	&\le&\lim_{\varepsilon\to 0}\varepsilon ||D\varphi||_{L^\infty(\Omega')} \int_{\mathrm{supp}\varphi}(1+|Dv_\varepsilon|^2)^{\frac{q-1}{2}}\,dx=0
\end{eqnarray*}
by virtue of the bound $q<\frac{np}{n-2}$ and of  estimate \eqref{dersec2c} and where we denoted by $\Omega'$ a compact set that contains the support of the test function $\varphi$. For what concerns $I_1$ we have by the definition of $\tilde{\mathcal{A}}_\varepsilon$ that
\begin{eqnarray*}
\lim_{\varepsilon\to 0}I_1^\varepsilon&\le& \lim_{\varepsilon\to 0}||D\varphi||_{L^\infty(\Omega')}\int_{\Omega'}|\mathcal{A}(x,Dv)-\tilde{\mathcal{A}}_\varepsilon(x,Dv)|\,dx\cr\cr	
&\le&\lim_{\varepsilon\to 0}||D\varphi||_{L^\infty(\Omega')}\int_{\Omega'}\left|\mathcal{A}(x,Dv)-\int_{B_1(0)}\phi(\omega)\mathcal{A}(x+\varepsilon\omega,Dv)\,d\omega\right|\,dx
\cr\cr
&\le&\lim_{\varepsilon\to 0}||D\varphi||_{L^\infty(\Omega'))}\int_{\Omega'}\int_{B_1(0)}\phi(\omega)\Big|\mathcal{A}(x,Dv)-\mathcal{A}(x+\varepsilon\omega,Dv)\Big|\,d\omega\,dx\cr\cr
&\le& \lim_{\varepsilon\to 0}\varepsilon||D\varphi||_{L^\infty(\Omega')}\int_{\Omega'}\int_{B_1(0)}\phi(\omega)k(x)+k(x+\varepsilon\omega)(1+|Dv|^2)^{\frac{q-1}{2}}\Big|\,d\omega\,dx\cr\cr
&\le& \lim_{\varepsilon\to 0}\varepsilon||D\varphi||_{L^\infty(\Omega')}\int_{\Omega'}(k(x)+k_\varepsilon(x))(1+|Dv|^2)^{\frac{q-1}{2}}\,dx\cr\cr
&\le& \lim_{\varepsilon\to 0}\varepsilon||D\varphi||_{L^\infty(\Omega')}\left(\int_{\Omega'}(k(x)+k_\varepsilon(x))^r\right)^{\frac{1}{r}}\left(\int_{\Omega'}(1+|Dv|^2)^{\frac{q-1}{2}\frac{r}{r-1}}\,dx\right)^{\frac{r-1}{r}}=0
\end{eqnarray*}
since $k_\varepsilon\to k$ strongly in $L^{r}(\Omega)$, by the bound $$\frac{q}{2}\frac{r}{r-1}<\frac{qr}{r-2}<\frac{2^*p}{2}$$
and again by \eqref{dersec2b}. For the estimate of $I_2^\varepsilon$, we observe that
\begin{eqnarray*}
	\lim_{\varepsilon\to 0}I_2^\varepsilon&\le&\lim_{\varepsilon\to 0}||D\varphi||_{L^\infty(\Omega')}\int_{\Omega'}|\mathcal{A}_\varepsilon(x,Dv)-\mathcal{A}_\varepsilon(x,Dv_\varepsilon)|\,dx\cr\cr
	&\le& \lim_{\varepsilon\to 0}||D\varphi||_{L^\infty(\Omega')}\int_{\Omega'}\int_{B_1(0)}\phi(\omega)\Big|\mathcal{A}(x+\varepsilon\omega,Dv)-\mathcal{A}(x+\varepsilon\omega,Dv_\varepsilon)\Big|\,d\omega\,dx\cr\cr
	&\le& L\lim_{\varepsilon\to 0}||D\varphi||_{L^\infty(\Omega')}\int_{\Omega'}|Dv-Dv_\varepsilon|(\mu^2+|Dv|^2+|Dv_\varepsilon|^2)^{\frac{p-2}{2}}\,dx\cr\cr
	&&+ c\varepsilon\lim_{\varepsilon\to 0}||D\varphi||_{L^\infty(\Omega')}\int_{\Omega'}|Dv-Dv_\varepsilon|(1+|Dv|^2+|Dv_\varepsilon|^2)^{\frac{q-2}{2}}\,dx\cr\cr
	&\le& L\lim_{\varepsilon\to 0}||D\varphi||_{L^\infty(\Omega')}\left(\int_{\Omega'}|Dv-Dv_\varepsilon|^p\,dx\right)^{\frac{1}{p}}\cr\cr
	&&\qquad\cdot\left(\int_{\Omega'}(1+|Dv|^2+|Dv_\varepsilon|^2)^{\frac{q-2}{2}\frac{p}{p-1}}\,dx\right)^{\frac{p-1}{p}}=0,
	\end{eqnarray*}
	since $Dv_\varepsilon\to Dv$ strongly in $L^{p}(\mathrm{supp}\varphi)$ and by virtue of the bound $$\frac{q-2}{2}\frac{p}{p-1}< \frac{q-1}{2}\frac{p}{p-1}<\frac{np}{n-2}$$
	and again by \eqref{dersec2b}.
	With these estimates at our disposal, we conclude that for every $\varphi\in C^1_c(\Omega)$
	$$\int_{\Omega}\langle\mathcal{A}(x,Dv),D\varphi\rangle\,dx=\int_{\Omega}\langle b,D\varphi\rangle\,dx,$$
	then $v\in (u_0+W^{1,p}_0(\Omega))\cap W^{2,p}_{\mathrm{loc}}(\Omega)$ is a  solution to \eqref{equaapp}.

\vskip2cm

\end{document}